\documentclass[conference]{ieeeconf}

\IEEEoverridecommandlockouts                              %
\usepackage[utf8]{inputenc}
\usepackage[T1]{fontenc}
\usepackage{amsmath}
\usepackage{amsfonts}
\usepackage{svg}
\usepackage{amssymb}
\usepackage{xcolor}
\usepackage{lettrine}
\usepackage{bm}
\usepackage{hyperref}
\usepackage{mathtools}
\usepackage{float}
\usepackage{cite}
\usepackage{algorithm}
\usepackage{algpseudocode}
\DeclareMathOperator{\tr}{tr}

\newtheorem{lemma}{Lemma}[section]
\newtheorem{theorem}{Theorem}[section]
\newtheorem{corollary}{Corollary}[section]
\newtheorem{remark}{Remark}[section]

\newtheorem{definition}{Definition}[section]
\newtheorem{assumption}{Assumption}[section]
\newtheorem{proposition}{Proposition}[section]

\title{\LARGE \bf
A Primal-Dual Gradient Descent Approach to the \\ Connectivity Constrained Sensor Coverage Problem
}

\author{Mathias B. Agerman$^{1}$, Ziqiao Zhang$^{2}$, Jong Gwang Kim$^{3}$, Shreyas Sundaram$^{2}$ and Christopher G. Brinton$^{2}$ %
\thanks{$^{1}$ M. Agerman is a student of Engineering Mathematics,
        KTH Royal Institute of Technology, SE-100 44 Stockholm, Sweden.
        Email: mathba@kth.se.}%
\thanks{$^{2}$ Z. Zhang, S. Sundaram, and C. Brinton are with
the Elmore Family School of Electrical and Computer Engineering, Purdue University, West
Lafayette, IN 47906, USA. Email: \{zhan5780, sundara2, cgb\}@purdue.edu.}%
\thanks{$^{3}$ J. Kim is with the Coles College of Business, Kennesaw State University, GA 30144, USA. Email: jkim311@kennesaw.edu}%
\thanks{This work was supported by Saab Inc. and Saab AB.}%
}

\begin{document}

\maketitle
\thispagestyle{empty}
\pagestyle{empty}

\begin{abstract}

Sensor networks play a critical role in many situational awareness applications. %
In this paper, we study the problem of determining sensor placements to balance coverage and connectivity objectives  over a target region. Leveraging algebraic graph theory, we formulate a novel optimization problem to maximize sensor coverage over a spatial probability density of event likelihoods while adhering to connectivity constraints. 
To handle the resulting non-convexity under constraints, we develop an augmented Lagrangian-based gradient descent algorithm inspired by recent approaches to efficiently identify points satisfying the Karush-Kuhn-Tucker (KKT) conditions. 
We establish convergence guarantees by showing necessary assumptions are satisfied in our setup, including employing Mangasarian-Fromowitz constraint qualification to prove the existence of a KKT point.
Numerical simulations under different probability densities demonstrate that the optimized sensor networks effectively cover high-priority regions while satisfying desired connectivity constraints. 
\end{abstract}

\section{INTRODUCTION}

We consider the problem of deploying a network of static sensors over a target region, given a geo-spatial distribution of event likelihoods. In situational awareness applications such as Intelligence, Surveillance, and Reconnaissance (ISR), it is desirable to maximize coverage over such regions (i.e., to make sure critical events are detected) while ensuring the network is connected (i.e., so that critical information propagates through the system).
Combining coverage and connectivity objectives is non-trivial due to their contradictory nature: intuitively, maximizing coverage typically requires the sensor nodes to spread apart, while connectivity maintenance requires them to stay sufficiently close together, depending on the number of sensors available. In this work, we aim to provide a computationally efficient solution to balancing connectivity and coverage while making minimal assumptions on the deployment region.

\subsection{Related Work}
Node placement has been well-studied in the literature on wireless sensor networks \cite{survey_coverage_connectivity, connectivity_survey, survey_coverage_connectivity_2}
In general, there are two main classes of sensor deployment algorithms: optimization-based methods and control-theoretic methods. While the latter methods are primarily used to guide mobile sensors to appropriate locations, they can also be used to deploy a static network by simulating the trajectories of the sensors and placing the sensors at the final points. 

\subsubsection{Optimization-based methods}
Optimization-based algorithms usually adopt a boolean disk coverage model \cite{GA, SA}. The coverage problem is NP-hard, which is often solved heuristically using methods such as genetic algorithms \cite{GA} or simulated annealing \cite{SA}. The problem is often discretized to account for the binary notion of connectivity \cite{discrete, discrete2}. In this paper, we relax the binary definition of connectivity, as in \cite{connectivity}, and use algebraic graph theory to model the communication topology. To ensure the network is connected, we incorporate a constraint based on the \textit{algebraic connectivity} of the network graph. We develop a primal-dual gradient descent algorithm based on \cite{ppala} to guarantee convergence to a point satisfying the Karush-Kuhn-Tucker (KKT) conditions. 

In this respect, the authors in \cite{connectivity} derived a potential field approach to guarantee connectivity by relaxing the definition of connectivity. Although this approach is successful when combined with other binary objectives such as collision avoidance and tracking leaders, it cannot be easily modified to deal with the maximization of other objectives. 
Our method is amenable to such modifications, e.g., we will augment our objective function with a regularization term to reflect prior knowledge of a desirable network. 

\subsubsection{Control-theoretic methods}
Control-theoretic algorithms have often studied \textit{expected-value multicenter functions} as coverage objectives, where strength decreases based on sensor distance \cite{basis, cortes2004coverage, distributed, generalized_voronoi} with possible discontinuities \cite{bullo_book}. 
The advantage of these models are their distributive nature, making them suitable for applications in mobile sensing networks where communication between agents cannot always be guaranteed. To maintain connectivity while achieving good coverage, a few works have incorporated algebraic connectivity models as well. For instance, \cite{control_lambda_2} derived a decentralized control law based on the algebraic connectivity to ensure the mobile sensor network remains connected. The control law was later applied in \cite{control_lambda_2_implementation} to balance coverage and connectivity objectives. In \cite{similar_problem}, a control law that ensures connectivity was proposed together with regularization to drive agents out of local minima. In \cite{control_barrier, cbf}, constraints for \textit{connectivity maintenance} were derived based on the algebraic connectivity. This approach was also used in \cite{voronoi_with_connectivity} to induce minimal change to the primary coverage controller to improve coverage while ensuring the control law still maintains connectivity among mobile sensors.

On the other hand, if the aim is to deploy a network of \textit{static} sensors, the control-theoretic algorithms proposed in these works have two main drawbacks: 1) it is generally assumed that the network is initially connected, and 2) to obtain formal guarantees on coverage performance while ensuring connectivity, the control law must generally be constrained by the current graph topology (or parts of it). While these assumptions are well motivated in mobile sensor networks, they can cause sensors to get stuck at undesirable solutions which are important to avoid for long-term deployments of static sensor networks. Our proposed algorithm, by contrast, allows the sensors to converge from any initial configuration and treats connectivity as a global attribute of the network.

\subsection{Summary of Contributions}
In this work, we develop a novel methodology for solving the sensor placement problem to balance flexible connectivity and coverage objectives over a target region. Compared to existing control-theoretic approaches, we embrace the centralized nature of network planning and, by doing so, are able to propose efficient optimization algorithms that find solutions over a geo-spatial distribution of event likelihoods. Contrary to existing optimization-based approaches, we adopt expected-value multicenter function objectives and a relaxed notion of connectivity which yields strong theoretical guarantees. Specifically, our main contributions are: 
\begin{itemize}
    \item[(a)] Developing a primal-dual gradient descent-based algorithm for optimizing sensor placements to balance coverage and connectivity objectives over a target region, while incorporating regularization to adjust the influence of prior knowledge distributions.
    \item[(b)] Establishing convergence guarantees of our algorithm to a KKT-point from any initial configuration, by showing necessary assumptions such as Mangasarian-Fromowitz Constraint Qualification are satisfied in our setup.
    \item[(c)] Numerically validating the convergence behavior and constraint satisfaction of sensor placements obtained by our algorithm under different geo-spatial event likelihood distributions.
\end{itemize}

\section{MODELING PRELIMINARIES}

In this section, we briefly present some key concepts of graph connectivity that will be leveraged throughout the paper. For a more thorough introduction, see e.g., \cite{connectivity}. 

Assume there are $n$ nodes with positions $x_1, x_2, \dots, x_n \in \mathbb{R}^d$ and let $\bm{x} = [x_1^T \ x_2^T \ \dots \ x_n^T]^T \in \mathbb{R}^{nd}$. The $k$-th component of an arbitrary node's position $x_i$ is denoted as $x_i^{(k)}$, where $k=1,\dots, d$.  Moreover, let $G(\bm{x})$ be the graph with vertices $V = \{1, 2, \dots, n\}$, indexed by the respective nodes. We define the edge set to be 
\begin{equation}\label{def:edge_set}
    E(\bm{x}) = \left\{ (i, j) \in V \times V : ||x_i - x_j|| \leq r, \ i \neq j \right\},
\end{equation}
where the norm $||\cdot||$ is the Euclidean distance and $r > 0$ is the transmission range. We say that $G$ is \textit{connected} if there exists a path between any two vertices of the graph. Note that \eqref{def:edge_set} assumes a boolean disk communication model between nodes, making the notion of connectivity a discontinuous quality of the network. 
To enable the application of gradient descent methods, we relax this definition by introducing smooth edge weights as follows.  %

\begin{definition}\label{def:adjacency}
    The \textit{weighted adjacency matrix}, $A(\bm{x}) \in \mathbb{R}^{n\times n}$, of a graph $G(\bm{x})$ is defined as
    \begin{equation*}
        A(\bm{x}) = (a_{ij}(\bm{x})) = \begin{cases}
            0, & i = j, \\
            \sigma_w\left( \varepsilon - ||x_i - x_j||\right), & i \neq j,
        \end{cases}
    \end{equation*}
    where $\sigma_w(y) = 1/(1+e^{-wy})$ is the sigmoid function, $w > 0$ determines the steepness of the weight decay and $\varepsilon>0$.
\end{definition}

In general, the results of this paper hold for any sufficiently smooth strictly increasing function $\sigma_w$. As $w$ tends to infinity, the above function approximates the binary nature of the edge set in \eqref{def:edge_set}, with $\varepsilon$ representing the transmission range. %
Based on this continuous interpretation of edge weights, \cite{connectivity} characterizes the notion of connectivity through the spectral properties of its Laplacian matrix, thus facilitating a framework better suited to gradient-based methods.
\vspace{0.05in}
\begin{definition}
    The \textit{weighted Laplacian matrix}, $\mathcal{L}(\bm{x}) \in \mathbb{R}^{n\times n}$, of a graph $G(\bm{x})$ is defined as
    \begin{equation*}
        \mathcal{L}(\bm{x}) = \Delta(\bm{x}) - A(\bm{x}),
    \end{equation*}
    where $\Delta(\bm{x}) = \text{diag}\left(\sum_{j=1}^n a_{ij}(\bm{x})\right)$ is the \textit{valency matrix} of $G(\bm{x})$.
\end{definition}
\vspace{0.05in}
The following lemmas are crucial for the remainder of the paper.
\vspace{0.05in}

\begin{lemma}[\cite{laplacian}]\label{lemma:laplacian}
    Let $\lambda_1(\bm{x}) \leq \lambda_2(\bm{x}) \leq \dots \leq \lambda_n(\bm{x})$ be the eigenvalues of the Laplacian $\mathcal{L}(\bm{x})$. Then
    \begin{itemize}
        \item[$(i)$] $\lambda_1(\bm{x}) = 0$ and the corresponding eigenvector is the all-ones vector $\bm{1}_n \in \mathbb{R}^n$, and
        \item[$(ii)$] $\lambda_2(\bm{x}) > 0$ if and only if $G(\bm{x})$ is connected.
    \end{itemize}
\end{lemma}
\vspace{0.05in}

Unfortunately, as remarked in \cite{connectivity}, $\lambda_2(\bm{x})$ is not differentiable. Therefore, the following lemma will be of  use.
\vspace{0.05in}

\begin{lemma}[\cite{connectivity}]\label{lemma:P}
    Let $P = [p_1 \ p_2 \ \dots \ p_{n-1} ] \in \mathbb{R}^{n \times (n-1)}$ be a matrix such that $p_i^Tp_j = 0$ for all $i \neq j$ and $p_i^T \bm{1}_n = 0$ $\forall i$. Then, with the same notation as in Lemma \ref{lemma:laplacian}, $\lambda_2(\bm{x}), \dots, \lambda_n(\bm{x})$ are the eigenvalues of $P^T\mathcal{L}(\bm{x})P$. In particular, $G(\bm{x})$ is connected if and only if 
    \begin{equation}
        \det{\left(P^T\mathcal{L}(\bm{x})P\right)} > 0. \notag
    \end{equation}
    
\end{lemma}
\vspace{0.1in}

\begin{remark}
    The matrix $P$ can, for instance, be obtained by applying Gram-Schmidt orthonormalization on the columns of the matrix $I_n - \frac{1}{n}\bm{1}\bm{1}^T$, where $I_n$ is the $n \times n$ identity matrix. 
\end{remark}

Lastly, we have the following closed-form expression for the gradient of $M(\bm{x}) := P^T\mathcal{L}(\bm{x})P$.
\vspace{0.05in}

\begin{lemma}[\cite{connectivity}]\label{lemma:grad_g}
    With $M(\bm{x})$ defined as above and $C(\bm{x})$ defined as the corresponding cofactor matrix, the following holds for all components $x_i^{(k)}$ of $\bm{x}$, $k=1,\cdots, d$, 
    \begin{equation}\label{eq:trace}
    \begin{aligned}
        \frac{\partial}{\partial x_i^{(k)}} \det \, &M(\bm{x}) = \tr \left[ C^T(\bm{x}) \frac{\partial}{\partial x_i^{(k)}} M(\bm{x})\right]\\
        &= \det M(\bm{x})\tr\left[M^{-1}(\bm{x})\frac{\partial}{\partial x_i^{(k)}}M(\bm{x})\right],
    \end{aligned}
    \end{equation}
    where $\frac{\partial}{\partial x_i^{(k)}} M(\bm{x}) = P^T \frac{\partial}{\partial x_i^{(k)}}\mathcal{L}(\bm{x}) P$ and $\frac{\partial}{\partial x_i^{(k)}} \mathcal{L}(\bm{x})$ is element-wise differentiation of $\mathcal{L}(\bm{x})$ with respect to $x_i^{(k)}$.
\end{lemma}

The second equality of \eqref{eq:trace} assumes that $M(\bm{x})$ is invertible and, while this is a theoretical consequence of the smooth edge weights introduced in Definition \ref{def:adjacency}, the matrix $M(\bm{x})$ can be ill-conditioned in practice. Therefore, although this form is useful from a theoretical standpoint, it is not recommended for numerical applications.

\section{PROBLEM FORMULATION}
In this section, we first present the unconstrained coverage problem, and then formulate the connectivity constrained coverage problem for static sensor networks.

\subsection{The Unconstrained Problem}

Let $Q \subset \mathbb{R}^d$ be a convex and compact set and $x_1, x_2, \dots, x_n \in Q$ denote the positions of the $n$ static sensors. Furthermore, let $\phi$ be a strictly positive probability density on $Q$ and $f: \mathbb{R}_{\geq 0} \to \mathbb{R}_+$ be a strictly increasing continuously differentiable function such that $f(||q - x_i||)$ quantifies the degree of uncertainty in a measurement at a point $q \in Q$ by the sensor at $x_i$, $\forall i \in V$. A commonly studied objective function (see e.g. \cite{bullo_book}) is the \textit{expected-value multicenter} function 
\begin{equation}\label{eq:unconstrained}
    \mathcal{H}(\bm{x})=\mathcal{H}(x_1, \dots, x_n) = \int_Q \min_i f(||q - x_i||) \phi(q) \ dq.
\end{equation}
By defining the Voronoi region of sensor $i$ as
\begin{equation}
    V_i = \{q \in Q: ||x_i - q|| \leq ||x_j - q||, \ \forall j \neq i\},
\end{equation}
it is clear that $\eqref{eq:unconstrained}$ can be equivalently expressed as
\begin{equation}
    \mathcal{H}(\bm{x})=\mathcal{H}(x_1, \dots, x_n) =  \sum_{i=1}^n \int_{V_i} f(||q - x_i||) \phi(q) \ dq.
\end{equation}
The strength of this formulation, usually leveraged in mobile sensor networks, is its distributive nature. In particular, one can show \cite{bullo_book} that whenever $x_i \neq x_j$ for all $i \neq j$
\begin{equation}\label{eq:grad_H}
    \nabla_{x_i}{\mathcal{H}} = \int_{V_i} \nabla_{x_i}{f(||q-x_i||)} \phi(q) \ dq. 
\end{equation}

\subsection{The Connectivity Constrained Problem}

A weakness of formulation \eqref{eq:unconstrained} is that it does not consider connectivity among sensors. %
Therefore, we consider problem \eqref{eq:unconstrained} coupled with a connectivity constraint. To this end, let $\bm{x} = [x_1^T \ x_2^T \ \dots \ x_n^T]^T \in \mathbb{R}^{nd}$ be the $nd \times 1$ vector obtained by stacking $x_1, \dots, x_n$ on top of each other and $\mathcal{L}(\bm{x})$ be the weighted Laplacian matrix of the graph $G(\bm{x})$ induced by a sensor placement. In light of Lemma \ref{lemma:P} we consider the following constrained problem
\begin{equation}\label{problem:constrained}
\begin{aligned}
\min_{\bm{x} \in Q^n} \quad & \mathcal{H}(\bm{x})\\
\textrm{s.t.} \quad & \tau - \det{\left(P^T\mathcal{L}(\bm{x})P\right)} \leq 0,\\
&\delta - \| x_i - x_j\| \leq 0, \ 1 \leq i < j \leq n.
\end{aligned}
\end{equation}
where $\tau > 0$ is a (small) threshold to enforce connectivity on $G(\bm{x})$ and $\delta > 0$ is the minimum distance between sensors. To simplify notation, we also define $g(\bm{x}) := \tau - \det{\left(P^T\mathcal{L}(\bm{x})P\right)}$ and $g_{ij}(\bm{x}) := \delta - \|x_i - x_j\|$, for $1 \leq i < j \leq n$. %
As discussed later, the minimum distance constraints are introduced to ensure the compactness of the feasible region and are also motivated by the sensors having some physical radius.

Problem \eqref{problem:constrained} is a non-convex and nonlinear optimization problem. Its strength, however, lies in the fact that $\mathcal{H}$ and $g$ are continuously differentiable as long as $x_i \neq x_j$ for all $i \neq j$, making gradient descent methods a suitable approach to tackle it. 

To further highlight the flexibility of the proposed problem, we note that there are algorithms capable of solving problems such as \eqref{problem:constrained} even with the addition of a regularization term. Formally, if $r: \mathbb{R}^{nd} \to \mathbb{R} \cup \{ + \infty \}$ is a convex, closed, proper, and possibly non-smooth function (see e.g. \cite{ppala}), we may also consider
\begin{equation}\label{problem:regularized}
\begin{aligned}
\min_{\bm{x} \in Q^n} \quad & \mathcal{H}(\bm{x}) + r(\bm{x})\\
\textrm{s.t.} \quad & \tau - \det{\left(P^T\mathcal{L}(\bm{x})P\right)} \leq 0,\\
&\delta - \| x_i - x_j\| \leq 0, \ 1 \leq i < j \leq n.
\end{aligned}
\end{equation}
Incorporating regularization such as the one mentioned above is, in general, a non-trivial task.  Certain approaches such as \cite{similar_problem} are able to give some locally optimal guarantees with some classes of regularization functions, but cannot be easily modified to incorporate more general classes (such as non-smooth regularization). 

\section{SOLUTION METHODOLOGY}

Since problem \eqref{problem:constrained} is a constrained non-convex optimization problem, it is generally challenging to find even an approximate global minimizer. Hence,  suitable convergence criteria are required to provide reasonable computational guarantees. As a practical alternative, the goal is often to identify a KKT-point (i.e., a first-order stationary point) \cite{Bertsekas}. 

In this work, we will use the \textit{Proximal-Perturbed Augmented Lagrangian Algorithm} (PPALA) \cite{ppala} as shown in Algorithm \ref{alg: PPALA}, to find a KKT-point. 
The PPALA applies a gradient scheme to the \textit{Proximal-Perturbed Augmented Lagrangian} (PPAL):
\begin{equation*}
    \mathcal{L}_\rho(\bm{x}, \bm{u},\bm{z},\bm{\lambda}, \bm{\mu}) = \ell_\rho(\bm{x}, \bm{u},\bm{z},\bm{\lambda}, \bm{\mu}) + r(\bm{x}),
\end{equation*}
where 
\begin{equation*}
    \begin{split}
    \ell_\rho(\cdot) = \mathcal{H}(\bm{x}) &+ \bm{\lambda}^T( \bar{g}(\bm{x}) + \bm{u} - \bm{z} ) + \bm{\mu}^T\bm{z} + \frac{\omega}{2}\|\bm{z}\|^2  \\
    &- \frac{\beta}{2}\|\bm{\lambda} - \bm{\mu}\|^2 + \frac{\rho}{2}\|\bar{g}(\bm{x}) + \bm{u}\|^2.
    \end{split}
\end{equation*}
State $\bm{x}$ is the position vector $\bm{x}$ in problem \eqref{problem:constrained}, $\bm{\lambda}$ is the Lagrange multiplier, $\bm{u}$ is the slack variable, $\bm{z}$ is the perturbation variable defined by $\bm{z}=\bar{g}(\bm{x})+\bm{u}$, $\bm{\mu}$ is the auxiliary multiplier, and $\bar{g}(\bm{x})$ is the collection of all inequality constraints.
We also define $\omega>0$ and $\rho = \frac{\omega}{1+\omega\beta}$ as penalty parameters, where $\beta>0$ is a dual proximal parameter. Moreover, let $U$ be an upper bound of $\|g(\bm{x})\|$. Lastly, we define step sizes $\eta > 0$, $\kappa > 0$ and $\sigma(t) > 0$.
\begin{algorithm}
    \caption{Proximal-Perturbed Augmented Lagrangian Algorithm (PPALA) \cite{ppala}}\label{alg: PPALA}
    \begin{algorithmic}[1]
        \State \textbf{Input:} Initialization $(\bm{x}(0), \bm{u}(0), \bm{z}(0), \bm{\lambda}(0), \bm{\mu}(0))$, and parameters $\omega > 1$, $\beta \in (0,1)$, $\rho = \frac{\omega}{1 + \omega \beta}$, and $T$.
        \For{$t = 0, 1, \ldots, T$}
            \State Compute $\bm{x}(t+1)$ by the proximal gradient scheme:
            \[
            \begin{aligned}
            &\bm{x}(t+1) = \mathop{\arg\min}_{\bm{x} \in Q^n} \{ \| \bm{x} - \bm{x}(t) \|^2/(2\eta) + r(\bm{x})  \\
            &+ 
            \langle \nabla_{\bm{x}} \ell_\rho(\bm{x}(t), \bm{u}(t), \bm{z}(t), \bm{\lambda}(t), \bm{\mu}(t)), \bm{x} - \bm{x}(t) \rangle\};
            \end{aligned}
            \]
            \State Compute $\bm{u}(t+1)$ by the projected gradient descent:
            \[
            \bm{u}(t+1) = \Pi_{[0,U]}[\bm{u}(t) - \kappa(\bm{\lambda}(t) + \rho\left(\bar{g}(\bm{x}(t)) + \bm{u}(t)\right))];
            \]
            \State Update the auxiliary multiplier $\bm{\mu}(t+1)$ by:
            \[
            \bm{\mu}(t+1) = \bm{\mu}(t) + \sigma(t)(\bm{\lambda}(t) - \bm{\mu}(t));
            \]
            \State Update the multiplier $\bm{\lambda}(t+1)$ by:
            \[
            \bm{\lambda}(t+1) = \bm{\mu}(t+1) + \rho(\bar{g}(\bm{x}(t+1)) + \bm{u}(t+1));
            \]
            \State Compute $\bm{z}(t+1)$ by:
            \[
            \bm{z}(t+1) = \frac{1}{\omega}(\bm{\lambda}(t+1) - \bm{\mu}(t+1));
            \]
        \EndFor
    \end{algorithmic}
\end{algorithm}

The PPALA guarantees the convergence to a KKT-point in the feasible region $\mathcal{D} \subset Q^n$ of problem \eqref{problem:constrained} under the following standard assumptions.
\vspace{0.1in}

\begin{assumption}\label{assumption:KKT}
    There exists a point $(\bm{x}, \bm{\lambda}) \in \mathcal{D} \times \mathbb{R}^{m}$ satisfying the KKT conditions, where $\bm{\lambda}$ is the Lagrange multiplier associated with the $m = \frac{n(n-1)}{2} + 1$ inequality constraints in \eqref{problem:constrained}.
\end{assumption}
\vspace{0.1in}

\begin{assumption}\label{assumption:compact}
    Region $\mathcal{D}$ is compact.
\end{assumption}
\vspace{0.1in}

\begin{assumption}\label{assumption:Lipschitz}
    Functions $\nabla \mathcal{H}$ and $\nabla g$ are Lipschitz continuous on the domain $\mathcal{D}$. That is, there exist constants $L_{\mathcal{H}}>0$ and $ L_g > 0$ such that
    \begin{equation*}
    \begin{aligned}
        ||\nabla \mathcal{H}(\bm{q}) - \nabla \mathcal{H}(\bm{q'})|| &\leq L_{\mathcal{H}}||\bm{q}-\bm{q'}||, \quad&\forall \bm{q},\bm{q'} \in \mathcal{D},\\
        ||\nabla g(\bm{q}) - \nabla g(\bm{q'})|| &\leq L_g||\bm{q}-\bm{q'}||, \quad&\forall \bm{q},\bm{q'} \in \mathcal{D}.
    \end{aligned}
    \end{equation*}
\end{assumption}
\vspace{0.15in}

We have introduced the domain $\mathcal{D}$ since $\mathcal{H}$ is only differentiable on $Q^n \setminus S$ where 
\begin{equation*}
    S = \{ [x_1^T \ x_2^T \ \dots \ x_n^T]^T \in Q^n : \exists i \neq j \text{ s.t. } x_i = x_j\}.
\end{equation*}
In practice, it seems possible to neglect this subtlety and still achieve convergence, but we will nonetheless treat this rigorously in the upcoming section for theoretical completeness. Lastly, if one were to remove the minimum-distance constraints and choose $\mathcal{D} = Q^n \setminus S$,  Assumption \ref{assumption:compact} would be violated since $S$ is closed, and thus ensuring that all three assumptions are satisfied is non-trivial. 

In order to justify Assumption \ref{assumption:KKT}, we need to show that our problem setting satisfies certain constraint qualifications, which will be the focus of the next section.

\section{CONVERGENCE ANALYSIS}

In this section, we will show that Assumptions \ref{assumption:KKT} - \ref{assumption:Lipschitz} are satisfied for problem \eqref{problem:constrained} under the mild assumption that the minimum distance between sensors can be chosen sufficiently small. This will then guarantee that the PPALA guarantees the convergence to a KKT-point by \cite{ppala}. 

\subsection{Existence of a KKT-point}

The KKT-conditions are a set of first-order optimality conditions commonly used to locally solve non-convex problems \cite{Bertsekas}. In order to guarantee that the KKT-conditions are necessary at a minimizer, one usually has to show a constraint qualification, i.e. that the constraints are well-behaved around the minimizer. Therefore, in order to show the existence of a KKT-point for problem \eqref{problem:constrained}, i.e. Assumption \ref{assumption:KKT}, it suffices to establish a suitable constraint qualification at some minimizer, since the KKT conditions are then necessary at that point \cite{MFCQ}. In our case, we will show the Mangasarian-Fromowitz Constraint Qualification (described below). When analyzing constraint qualifications, an important concept is that of \textit{active} inequality constraints. An inequality constraint is called active if it is satisfied with equality.  

\begin{definition}[\cite{MFCQ}]\label{def:mfcq}
    A point $\bm{x} \in \mathcal{D}$ is said to satisfy the Mangasarian-Fromowitz Constraint Qualification (MFCQ) if there exists a vector $\bm{d} \in \mathbb{R}^{nd}$ such that
    \begin{align}
        g(\bm{x}) = 0 & \Rightarrow \nabla g(\bm{x})^T\bm{d} < 0, \\
        g_{ij}(\bm{x}) = 0 & \Rightarrow \nabla g_{ij}(\bm{x})^T\bm{d} < 0, \ \forall 1\leq i< j \leq n.
    \end{align}
\end{definition}
\vspace{0.1in}

Definition \ref{def:mfcq} essentially states that at point $\bm{x}$ there exists a direction $\bm{d}$ in which it is possible to move and make active constraints inactive. In order to show the MFCQ, we will use the following lemma.

\begin{lemma}[\cite{horn2012matrix}]\label{lemma:gershgorin} 
    (Gershgorin's Circle Theorem) Let $\Gamma = (\gamma_{ij}) \in \mathbb{C}^{n\times n}$ be a complex valued matrix. Let $R_i = \sum_{j \neq i} |\gamma_{ij}|$. Then every eigenvalue of $\Gamma$ is contained in at least one of the Gershgorin discs
    \begin{equation*}
        \bar{D}(\gamma_{ii}, R_i) = \left\{ z \in \mathbb{Z}: |z - \gamma_{ii}| \leq R_i \right\}, \ i = 1, \dots, n.
    \end{equation*}
\end{lemma}
\vspace{0.1in}

We are now ready to state one of the main theorems of this paper.

\begin{theorem}\label{thm:1}
    If $g(\bm{x}) = \tau - \det P^T\mathcal{L}(\bm{x})P = 0$, then $\nabla g(\bm{x}) \neq \bm{0}$. 
\end{theorem} 

\begin{proof}
    By Lemma \ref{lemma:grad_g} we have
    \begin{equation*}
        \frac{\partial}{\partial x_{l}^{(k)}} g(\bm{x}) = - \det(M(\bm{x}))\tr\left[ M^{-1}(\bm{x})P^T \frac{\partial}{\partial x_{l}^{(k)}}\mathcal{L}(\bm{x})P \right].
    \end{equation*}
    Since $\det(M(\bm{x}))= \det P^T\mathcal{L}(\bm{x})P = \tau > 0$ by assumption, it suffices to show that
    \begin{equation}\label{eq:nonzero_trace}
        \tr\left[ M^{-1}(\bm{x})P^T \frac{\partial}{\partial x_l^{(k)}}\mathcal{L}(\bm{x})P \right] \neq 0,
    \end{equation}
    for some indices $1 \leq l \leq n$, $1 \leq k \leq d$. Now, recalling that $\mathcal{L}(x) = \Delta(\bm{x}) - A(\bm{x})$, we obtain $\frac{\partial}{\partial x_l^{(k)}}\mathcal{L}(\bm{x})$ directly from the derivatives
    \begin{equation*}
        \frac{\partial}{\partial x_l^{(k)}}a_{ij}(\bm{x}) = \begin{cases}
            -\sigma_w'(\varepsilon - \|x_l - x_j\|)\frac{x_l^{(k)} - x_j^{(k)}}{\|x_l - x_j\|}, & \text{if } i = l,\\
            -\sigma_w'(\varepsilon - \|x_l - x_i\|)\frac{x_l^{(k)} - x_i^{(k)}}{\|x_l - x_i\|}, & \text{if } j = l,\\
            0, & \text{else.}
        \end{cases}
    \end{equation*}
    Because $x_1, \dots, x_n$ are distinct points there exist indices $1 \leq l \leq n$, $1 \leq k \leq d$ such that:
    \begin{align}
        x_{l}^{(k)} &\leq x_{j}^{(k)} \ \text{for all } j \neq l \label{eq:leq}, \\
        x_l^{(k)} &< x_j^{(k)} \  \text{for some } j \neq l \label{eq:lt}.
    \end{align}
    Since $\sigma_w' > 0$, it thus follows that $\frac{\partial}{\partial x_{l}^{(k)}}a_{ij}(\bm{x}) \geq 0$ for all $1 \leq i,j \leq n$ and is positive for at least some $i$ and $j$. Consequently, since $\frac{\partial}{\partial x_l^{(k)}}\mathcal{L}(\bm{x}) = \frac{\partial}{\partial x_l^{(k)}}\Delta(\bm{x}) - \frac{\partial}{\partial x_l^{(k)}}A(\bm{x})$, it follows that all of the Gershgorin discs are contained in the non-negative half-plane. Because $\frac{\partial}{\partial x_l^{(k)}}\mathcal{L}(\bm{x})$ is a real and symmetric matrix, it is positive semi-definite according to Lemma \ref{lemma:gershgorin}, and it is non-zero by \eqref{eq:lt}. Analogously to Lemma \ref{lemma:P}, $\frac{\partial}{\partial x_l^{(k)}}M(\bm{x}) = P^T\frac{\partial}{\partial x_l^{(k)}}\mathcal{L}(\bm{x})P$ is a positive semi-definite and non-zero matrix (i.e., it has at least one positive eigenvalue). Thus, there exists an orthogonal matrix $U$ such that $\frac{\partial}{\partial x_l^{(k)}}M(\bm{x}) = U^T\Lambda_{M'}U$, where $\Lambda_{M'} = \text{diag}(\nu_1, \nu_2, \dots, \nu_{n - 1})$ is the diagonal matrix with the eigenvalues of $\frac{\partial}{\partial x_l^{(k)}}M(\bm{x})$. By \eqref{eq:nonzero_trace} it suffices to show that $\tr\{M^{-1}(\bm{x})U^T\Lambda_{M'} U \} \neq 0$. We can obtain the following by the cyclic property of the trace
    \begin{equation*}
        \tr\{M^{-1}(\bm{x})U^T\Lambda_{M'} U \} = \tr\{UM^{-1}(\bm{x})U^T\Lambda_{M'}\}.
    \end{equation*}
    Finally, noting that $W = \left(w_{ij} \right) := UM^{-1}(\bm{x})U^T$ is positive definite since $M^{-1}$ is positive definite yields
    \begin{equation*}
        \tr\{ W\Lambda_{M'}\} = \sum_{i=1}^{n-1} \nu_i w_{ii} > 0,
    \end{equation*}
    since $w_{ii} > 0$ for all $1 \leq i \leq n-1 $ and $\nu_i > 0$ for at least some $i$.
\end{proof}

By the above proof, we have also proved the following corollary.

\begin{corollary}\label{corollary:hull_node}
    Let $x_i$ be some node satisfying \eqref{eq:leq} and \eqref{eq:lt} for some $1 \leq k \leq d$. Then there exists a $d_i \in \mathbb{R}^d$ such that $\nabla g(\bm{x})^T \bm{d} < 0$, where $\bm{d} = \begin{bmatrix}\bm{0}_1^T & \dots & \bm{0}_{i - 1}^T & d_i^T & \bm{0}_{i+1}^T & \dots & \bm{0}_n^T\end{bmatrix}^T$ and $\bm{0}_j$ is the zero vector in $\mathbb{R}^d$ for all $j \neq i$. 
\end{corollary}

Note that the existence of the node in Corollary \ref{corollary:hull_node} is always guaranteed whenever the nodes are distinct. To prove the MFCQ, we will further assume that we can choose the node such that it is not involved in any active minimum-distance constraints. While this assumption does not hold for all $\bm{x} \in \mathcal{D}$, it is a reasonable assumption for minimizers of \eqref{problem:constrained} since in maximum coverage problems we are not expecting sensors to neighbor each other by their physical radius in optimal configurations. In fact, keeping all other parameters fixed, the assumption is guaranteed to hold at a minimizer for some sufficiently small $\delta > 0$. Formally, if $G_\delta(\bm{x}) = (V, E_\delta(\bm{x}))$ is the graph with edges $(i, j) \in E_\delta$ if and only if $\|x_i - x_j\| = \delta$, we will assume that the node in Corollary \ref{corollary:hull_node} can be chosen to be isolated in $G_\delta(\bm{x})$. Graph $G_\delta(\bm{x})$ precisely captures the active minimum-distance constraints. 

\begin{theorem}\label{thm:2}
    Let $\bm{x} \in \mathcal{D}$ and let $x_i$ and $\bm{d}$ be those obtained from Corollary \ref{corollary:hull_node}. If $x_i$ is isolated in $G_\delta(\bm{x})$, then $\bm{x}$ satisfies the MFCQ. 
\end{theorem}

\begin{proof}
    By Corollary \ref{corollary:hull_node} $\nabla g(\bm{x})^T\bm{d} < 0$. Now, let $\Delta \bm{ d} = \begin{bmatrix}
        x_1^T & \dots & x_{i - 1}^T & \bm{0}_i^T & x_{i + 1}^T & \dots & x_n^T
    \end{bmatrix}^T$ and define $\bar{\bm{d}} = \bm{d} + t\Delta\bm{d}$, for some real number $t > 0$. Then, for any edge $(l, k) \in E_\delta(\bm{x})$ it can be shown that
    \begin{equation*}
    \begin{aligned}
        \nabla g_{lk}(\bm{x})^T\bar{\bm{d}} 
        &= \nabla g_{lk}(\bm{x})^T t\Delta\bm{d} \\
        &= \frac{(x_l - x_k)^T(tx_k - tx_l)}{\|x_l - x_k\|} \\ 
        &=-t\frac{\|x_l - x_k\|^2}{\|x_l-x_k\|} = -t\delta < 0.
    \end{aligned}
    \end{equation*}
    And lastly, 
    \begin{equation*}
        \nabla g(\bm{x})^T \bar{\bm{d}} = \underbrace{\nabla g(\bm{x})^T \bm{d}}_{<0} + t\left(\nabla g(\bm{x})^T \Delta\bm{d}\right) < 0
    \end{equation*}
    if $t$ is chosen sufficiently small. 
\end{proof}

\subsection{Compactness of the domain}

Showing the compactness of the feasible region is straightforward as it is the intersection of closed sets and $Q$ is assumed to be compact. Therefore, a formal proof is omitted. However, it is important to note that removing the minimum-distance constraints would contradict Assumption \ref{assumption:compact} since $S$ is closed and its complement $Q^n \setminus S$ is open. That being said, it appears possible in practice to disregard the minimum-distance constraints entirely and still achieve convergence.  

\subsection{Lipschitz continuity}

In this subsection, we show that $\nabla g$ and $\nabla\mathcal{H}$ are continuously differentiable on $Q^n \setminus S$. The latter will only be shown under some additional assumptions on $f$, which are satisfied at least by one of the most common choices in the literature, $f(x) = x^2/2$. From this, it follows from standard results that $\nabla g$ and $\nabla\mathcal{H}$ are Lipschitz-continuous on any compact subset $D \subset Q^n \setminus S$.  
\vspace{0.1in}

\begin{proposition}
    If $f$ is three times continuously differentiable and $f(0) = 0$, then $\nabla \mathcal{H}$ is continuously differentiable on $Q^n \setminus S$.
\end{proposition}

\begin{proof}
    It suffices to show that the partial derivatives of $\nabla \mathcal{H}$ exist and are continuous. Recalling \eqref{eq:grad_H} and using the chain rule, we can obtain
    \begin{equation}\label{eq:partial_k}
    \begin{aligned}
        \nabla_{x_i}\mathcal{H} &= \int_{V_i} \nabla_{x_i} {f(||q-x_i||)} \phi(q) \ dq \\
        &= \int_{V_i} \underbrace{\frac{f'(||q-x_i||)}{{||q-x_i||}}}_{=: h(||q-x_i||)} (x_i - q) \phi(q) \ dq .
    \end{aligned}
    \end{equation}
    Now, if $h(x) := f'(x)/x$ is continuously differentiable, the claim follows analogously to \textup{\cite[Theorem 2.16]{bullo_book}}, by applying the generalized Leibniz rule to every component of $\nabla \mathcal{H}$. Function $h$ is trivially continuously differentiable for nonzero $x$, but for $x = 0$ closer attention is required. By assumption and Taylor's formula $h(x) = f'(0)x + \frac{f''(0)}{2}x^2 + \mathcal{O}(x^3)$ around $x = 0$, we can have the following limit
    \begin{equation*}
        \lim_{\Delta x \to 0} \frac{h(\Delta x) - h(0)}{\Delta x} =  \lim_{\Delta x \to 0}\frac{f''(0)}{2} + \mathcal{O}\left((\Delta x)^2\right) = \frac{f''(0)}{2}, 
    \end{equation*}
    indicating that $h'(0)$ exists. For nonzero $x$ 
    \begin{equation*}
    \begin{aligned}
        h'(x) &= \frac{xf'(x)-f(x)}{x^2} \\
        &= \frac{f'(x) - f'(0)}{x} - \frac{f''(0)}{2} + \mathcal{O}(x).
    \end{aligned}
    \end{equation*}
    Thus, $\lim_{x\to 0} h'(x) = f''(0)/2$, i.e. $h'$ is continuous at $x = 0$, which means that $h$ is continuously differentiable. 
    Therefore, we conclude that $\nabla\mathcal{H}$ is continuously differentiable on $Q^n\setminus S$. 
\end{proof}
\vspace{0.1in}

\begin{proposition}
    $\nabla g$ is continuously differentiable on $Q^n \setminus S$.  
\end{proposition}

\begin{proof}
    Again, it suffices to show that the partial derivatives of $\nabla g$ exist and are continuous on $Q^n \setminus S$. Denote by $x_i^{(k)}$ an arbitrary component of $\bm{x}$ and consider $\frac{\partial}{\partial x_{i}^{(k)}} g$. By definition of $g$ and Lemma \ref{lemma:grad_g}, we have
    \begin{equation*}
        \frac{\partial}{\partial x_{i}^{(k)}} g = -\tr \underbrace{\left[C^T(\bm{x}) \frac{\partial}{\partial x_i^{(k)}} M(\bm{x})\right]}_{:= B(\bm{x})}
    \end{equation*}
    Since, by Definition \ref{def:adjacency}, the elements of $B(\bm{x})$ are just products and sums of smooth functions of $\bm{x}$ whenever $\bm{x} \in Q^n \setminus S$, the claim follows by the linearity of the trace. 
\end{proof}

In this section, we have thus shown that the PPALA converges to a point satisfying the KKT optimality conditions for problems \ref{problem:constrained} and \ref{problem:regularized}, ensuring they are locally optmal. 

\begin{figure}[t]
    \centering
    \includegraphics[width=\linewidth]{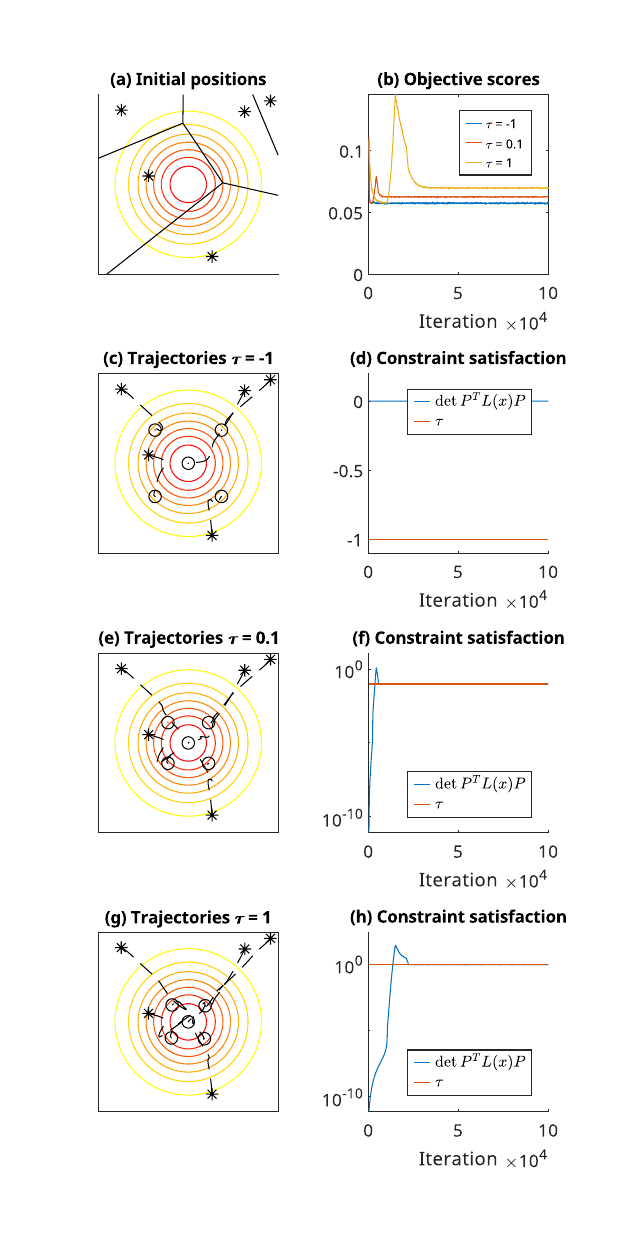}
    \vspace{-10mm}
    \caption{The deployment trajectories obtained when using the PPALA to solve problem \eqref{problem:constrained} with the density $\phi_1$ in \eqref{eq:unimodal} for $w = 20$ and $\varepsilon = 0.1$. Subfigures (c) - (d) capture the unconstrained problem with $\tau = -1$, while $\tau = 0.1$ and $\tau = 1$ in subfigures (e) - (f) and (g) - (h), respectively.}
    \label{fig:one_mode}
\end{figure}


\section{NUMERICAL RESULTS}

In this section, we present and discuss computational results for the PPALA algorithm applied to problem \eqref{problem:constrained}. 

\subsection{Simulation Setup}

The computational tests were implemented using \verb|MATLAB| on a laptop, and we used $\varepsilon = 0.1$, $w = 20$ and a few choices of the hyper-parameter $\tau>0$. In some of the simulations, we set $\tau = -1$ to represent the unconstrained problem; it is merely included to act as a reference of what a deployment without connectivity constraints would look like. Although the minimum-distance constraints were introduced to theoretically guarantee compactness of the feasible region and, in extension, convergence of the PPALA, they will be suppressed for simplicity in this section.  

We consider the case when $Q = [0, 1] \times [0, 1] \subset \mathbb{R}^2$, $f(x) = \frac{1}{2}x^2$ as in \cite{distributed, basis} and two choices of the density function $\phi$. First, we deploy $n = 5$ sensors over the density:
\begin{equation}\label{eq:unimodal}
    \phi_1(q) = \frac{\gamma_1}{\sigma\sqrt{2\pi}}e^{-\frac{||q-\mu_1||^2}{2\sigma^2}},
\end{equation}
where $\mu_1 = [0.5 \ 0.5]^T$, $\sigma = 1/5$ and $\gamma_1 > 0$ is a normalizing constant. The results of the deployment algorithm are shown in Fig. \ref{fig:one_mode}. 
Secondly, we deploy $n = 10$ sensors over the more complex density:
\begin{equation}\label{eq:bimodal}
    \phi_2(q) = \frac{\gamma_2}{\sigma\sqrt{2\pi}}\left( e^{-\frac{||q-\mu_2||^2}{2\sigma^2}} + e^{-\frac{||q-\mu_3||^2}{2\sigma^2}} \right),
\end{equation}
where $\mu_2 = [0.2 \ 0.2]^T$, $\mu_3 = [0.8 \ 0.8]^T$, and $\gamma_2 > 0$ is again a normalizing constant. The results of the deployment algorithm are shown in Fig. \ref{fig:bi_modal}. Lastly, the deployment algorithm was also tested in the setting of \eqref{problem:regularized} with the regularization function: 
\begin{equation}\label{eq:reg_instance}
    r(\bm{x}) = \frac{\alpha}{n}\sum_{i = 1}^n ||x_i - C(Q)||^2,
\end{equation}
where $\alpha > 0$ is a normalizing constant to balance the objectives and $C(Q)$ denotes the centroid of the domain $Q$. Similar regularization was used in \cite{similar_problem} to help drive agents out of poor local minima. The resulting deployments under regularization are shown in Fig. \ref{fig:regularized}, and they are also compared to an unregularized deployment. 


\begin{figure}[t]
    \centering
    \includegraphics[width=\linewidth]{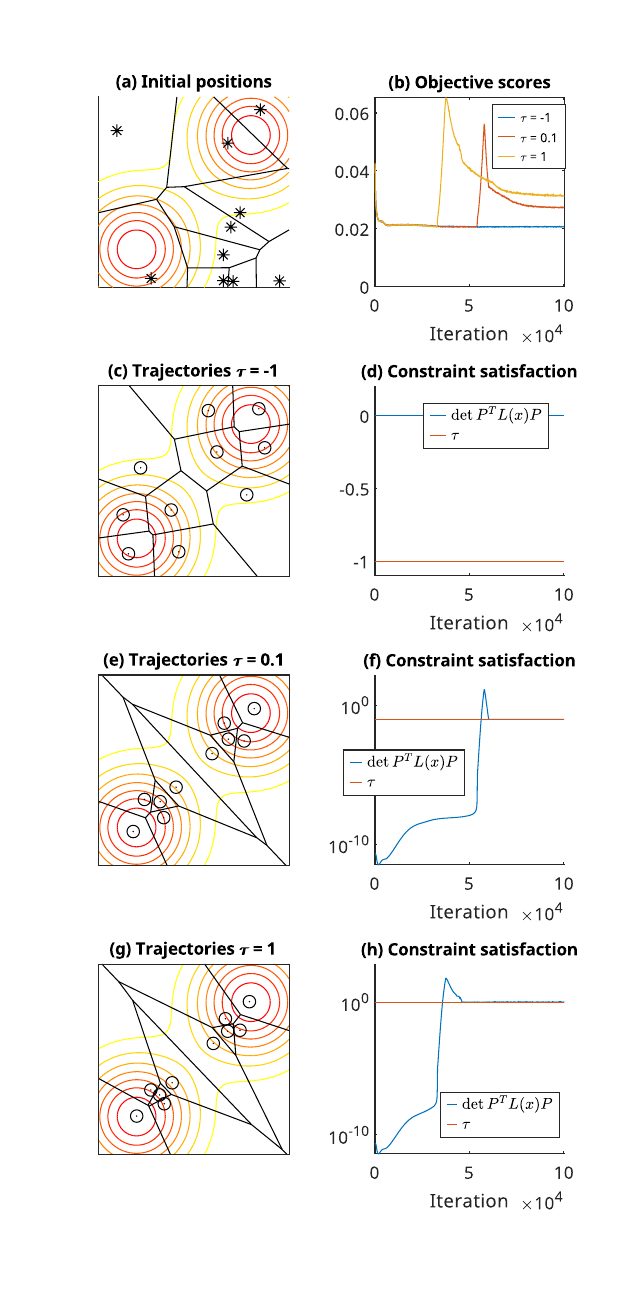}
    \vspace{-10.5mm}
    \caption{The deployments obtained when using the PPALA to solve problem \eqref{problem:constrained} with the bimodal density $\phi_2$ in \eqref{eq:bimodal} for $w = 20$ and $\varepsilon = 0.1$. Subfigures (c) - (d) capture the unconstrained problem with $\tau = -1$, while $\tau = 0.1$ and $\tau = 1$ in subfigures (e) - (f) and (g) - (h), respectively.}
    \label{fig:bi_modal}
\end{figure}

\subsection{Results and Discussion}

Fig. \ref{fig:one_mode} and Fig. \ref{fig:bi_modal} show that the PPALA can converge to local minima of the problem \eqref{problem:constrained} from a disconnected initial configuration. In Fig. \ref{fig:one_mode}, increasing the threshold value $\tau$, forces the nodes to be deployed with larger edge weights. Although this pattern can also be seen in Fig. \ref{fig:bi_modal}, we also notice that the distance between the two groups increases from Fig. \ref{fig:bi_modal}(e) to Fig.  \ref{fig:bi_modal}(g), despite an increase in $\tau$. This is to be expected since the algebraic connectivity can in fact still decrease as long as the remaining eigenvalues increase. However, as the decay of the edge weights $w$ increases and $\tau$ decreases, the resulting deployments will approximate the boolean disk communication model better and ensure the graph is connected (in the discrete regard). This does, naturally, pose greater numerical challenges, such as careful consideration of the step sizes used in the PPALA. Although this paper has established theoretical convergence guarantees even in these more challenging numerical settings, it is considered part of future work to explore these numerical challenges to reduce computational costs. 

Using Fig. \ref{fig:one_mode} one can also verify that the assumption made in Theorem \ref{thm:2} is not problematic in practice. Recall that, in order to establish the MFCQ, it was assumed that the node in Corollary \ref{corollary:hull_node} could be chosen to be isolated in $G_\delta(\bm{x})$. We also noted that for this assumption to hold at a minimizer, it suffices to choose the minimum distance between sensors, $\delta$, sufficiently small. To see this in practice, we consider the instance of the locally optimal deployment shown in Fig. \ref{fig:one_mode}(e). If $\delta$ were chosen to be smaller than the minimum distance between any pair of nodes in this deployment, the configuration would remain feasible and, as such, remain locally optimal even with the addition of the minimum-distance constraints. Moreover, due to how $\delta$ was chosen, the assumption of Theorem \ref{thm:2} would hold for the configuration. Thus, the configuration would be locally optimal and satisfy the MFCQ, ensuring that Assumption \ref{assumption:KKT} holds. 

Lastly, Fig. \ref{fig:regularized} provides an example of how regularization can complement the specified probability density. Increasing the regularization strength, $\alpha$, promotes deployments with denser coverage near the centroid of the domain. This can be seen as a more conservative approach, with regularization serving as a means to incorporate uncertainty into the prior belief about the underlying density function.

\begin{figure}[t]
    \centering
    \includegraphics[width=\linewidth]{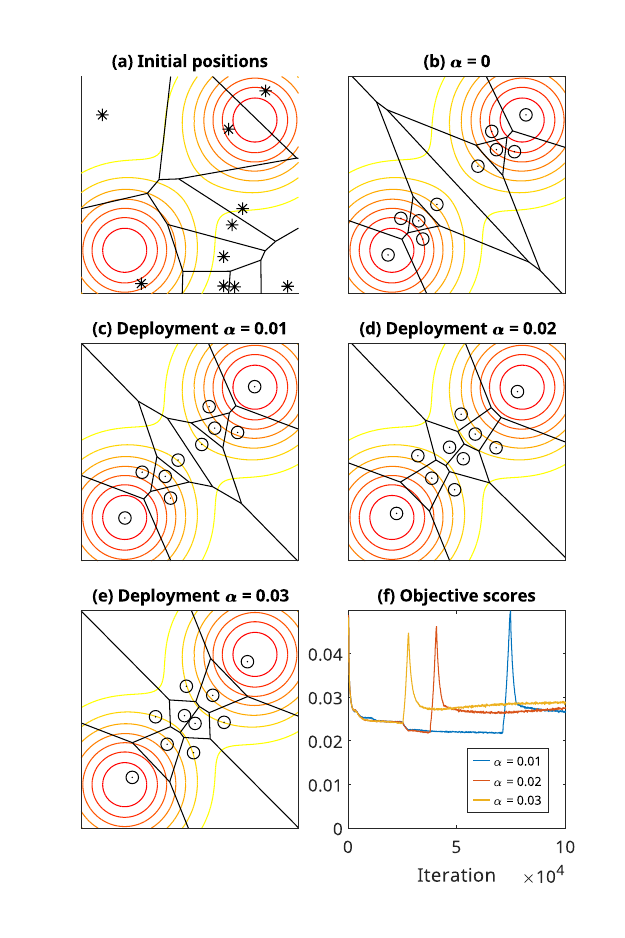}
    \vspace{-10mm}
    \caption{The deployments obtained from the PPALA when solving the regularized problem \eqref{problem:regularized} with the bimodal density $\phi_2$ in \eqref{eq:bimodal} for $w = 20$, $\varepsilon = 0.1$ and $\tau = 0.1$. Subfigure (b) captures the unregularized problem, while $\alpha = 0.01$, $\alpha = 0.02$ and $\alpha = 0.03$ in subfigures (c) - (e), respectively.}
    \label{fig:regularized}
\end{figure}
\vspace{0.1in}

\section{CONCLUSIONS}

In this work, we proposed a continuous optimization problem \eqref{problem:constrained} balancing the contradicting objectives of coverage and connectivity. The formulation lends itself to be solved by efficient non-convex optimization algorithms, and its flexibility is highlighted by the possibility of accounting for a wide range of regularization. The MFCQ was established, guaranteeing that the PPALA converges to a KKT-point. Moreover, numerical simulations showed that the PPALA can converge to feasible local minima from a disconnected initial configuration. Thus, the solution maintains good coverage while also achieving a higher degree of connectivity. The model deals with the trade-off between coverage and connectivity in a mathematically rigorous way. This does, however, require a smooth interpretation of the notion of connectivity. Therefore, a setting where a boolean transmission range is used (i.e., $w \to \infty$) will prove numerically challenging and, generally, careful consideration must be paid to the step size and other parameters of the PPALA to achieve convergence for large $w$. 

Recovering the discrete formulation by exploring these numerical challenges is part of future work. Still, one of the key strengths of the formulation is that it is well suited to gradient-based optimization approaches, and the fact that new efficient algorithms are emerging due to their application in machine learning promises good prospects for the proposed methodology.

\addtolength{\textheight}{0cm}   %

\bibliographystyle{IEEEtran}
\bibliography{refs}

\begin{thebibliography}{10}
\providecommand{\url}[1]{#1}
\csname url@samestyle\endcsname
\providecommand{\newblock}{\relax}
\providecommand{\bibinfo}[2]{#2}
\providecommand{\BIBentrySTDinterwordspacing}{\spaceskip=0pt\relax}
\providecommand{\BIBentryALTinterwordstretchfactor}{4}
\providecommand{\BIBentryALTinterwordspacing}{\spaceskip=\fontdimen2\font plus
\BIBentryALTinterwordstretchfactor\fontdimen3\font minus \fontdimen4\font\relax}
\providecommand{\BIBforeignlanguage}[2]{{%
\expandafter\ifx\csname l@#1\endcsname\relax
\typeout{** WARNING: IEEEtran.bst: No hyphenation pattern has been}%
\typeout{** loaded for the language `#1'. Using the pattern for}%
\typeout{** the default language instead.}%
\else
\language=\csname l@#1\endcsname
\fi
#2}}
\providecommand{\BIBdecl}{\relax}
\BIBdecl

\bibitem{survey_coverage_connectivity}
\BIBentryALTinterwordspacing
A.~Ghosh and S.~K. Das, ``Coverage and connectivity issues in wireless sensor networks: A survey,'' \emph{Pervasive and Mobile Computing}, vol.~4, no.~3, pp. 303--334, 2008. [Online]. Available: \url{https://www.sciencedirect.com/science/article/pii/S1574119208000187}
\BIBentrySTDinterwordspacing

\bibitem{connectivity_survey}
M.~M. Zavlanos, M.~B. Egerstedt, and G.~J. Pappas, ``Graph-theoretic connectivity control of mobile robot networks,'' \emph{Proceedings of the IEEE}, vol.~99, no.~9, pp. 1525--1540, 2011.

\bibitem{survey_coverage_connectivity_2}
\BIBentryALTinterwordspacing
C.~Zhu, C.~Zheng, L.~Shu, and G.~Han, ``A survey on coverage and connectivity issues in wireless sensor networks,'' \emph{Journal of Network and Computer Applications}, vol.~35, no.~2, pp. 619--632, 2012, simulation and Testbeds. [Online]. Available: \url{https://www.sciencedirect.com/science/article/pii/S1084804511002323}
\BIBentrySTDinterwordspacing

\bibitem{GA}
Y.~Yoon and Y.-H. Kim, ``An efficient genetic algorithm for maximum coverage deployment in wireless sensor networks,'' \emph{IEEE Transactions on Cybernetics}, vol.~43, no.~5, pp. 1473--1483, 2013.

\bibitem{SA}
\BIBentryALTinterwordspacing
N.~Coll, M.~Fort, and M.~Saus, ``Coverage area maximization with parallel simulated annealing,'' \emph{Expert Systems with Applications}, vol. 202, p. 117185, 2022. [Online]. Available: \url{https://www.sciencedirect.com/science/article/pii/S0957417422005723}
\BIBentrySTDinterwordspacing

\bibitem{discrete}
J.~N. Al-Karaki and A.~Gawanmeh, ``The optimal deployment, coverage, and connectivity problems in wireless sensor networks: Revisited,'' \emph{IEEE Access}, vol.~5, pp. 18\,051--18\,065, 2017.

\bibitem{discrete2}
\BIBentryALTinterwordspacing
M.~Mansour and F.~Jarray, ``An iterative solution for the coverage and connectivity problem in wireless sensor network,'' \emph{Procedia Computer Science}, vol.~63, pp. 494--498, 2015, the 6th International Conference on Emerging Ubiquitous Systems and Pervasive Networks (EUSPN 2015)/ The 5th International Conference on Current and Future Trends of Information and Communication Technologies in Healthcare (ICTH-2015)/ Affiliated Workshops. [Online]. Available: \url{https://www.sciencedirect.com/science/article/pii/S1877050915025090}
\BIBentrySTDinterwordspacing

\bibitem{connectivity}
M.~M. Zavlanos and G.~J. Pappas, ``Potential fields for maintaining connectivity of mobile networks,'' \emph{IEEE Transactions on Robotics}, vol.~23, no.~4, pp. 812--816, 2007.

\bibitem{ppala}
\BIBentryALTinterwordspacing
J.~G. Kim, A.~Chandra, A.~Hashemi, and C.~Brinton, ``A fast single-loop primal-dual algorithm for non-convex functional constrained optimization,'' 2024. [Online]. Available: \url{https://arxiv.org/abs/2406.17107}
\BIBentrySTDinterwordspacing

\bibitem{basis}
M.~Schwager, J.-J. Slotine, and D.~Rus, ``Decentralized, adaptive control for coverage with networked robots,'' in \emph{Proceedings 2007 IEEE International Conference on Robotics and Automation}, 2007, pp. 3289--3294.

\bibitem{cortes2004coverage}
J.~Cortes, S.~Martinez, T.~Karatas, and F.~Bullo, ``Coverage control for mobile sensing networks,'' \emph{IEEE Transactions on robotics and Automation}, vol.~20, no.~2, pp. 243--255, 2004.

\bibitem{distributed}
M.~Schwager, J.~McLurkin, and D.~Rus, ``Distributed coverage control with sensory feedback for networked robots.'' in \emph{robotics: science and systems}, 2006, pp. 49--56.

\bibitem{generalized_voronoi}
\BIBentryALTinterwordspacing
K.~R. Guruprasad and D.~Ghose, ``Coverage optimization using generalized voronoi partition,'' 2011. [Online]. Available: \url{https://arxiv.org/abs/0908.3565}
\BIBentrySTDinterwordspacing

\bibitem{bullo_book}
F.~Bullo, J.~Cort{\'e}s, and S.~Martinez, \emph{Distributed control of robotic networks: a mathematical approach to motion coordination algorithms}.\hskip 1em plus 0.5em minus 0.4em\relax Princeton University Press, 2009, vol.~27.

\bibitem{control_lambda_2}
\BIBentryALTinterwordspacing
L.~Sabattini, N.~Chopra, and C.~Secchi, ``Decentralized connectivity maintenance for cooperative control of mobile robotic systems,'' \emph{The International Journal of Robotics Research}, vol.~32, no.~12, pp. 1411--1423, 2013. [Online]. Available: \url{https://doi.org/10.1177/0278364913499085}
\BIBentrySTDinterwordspacing

\bibitem{control_lambda_2_implementation}
L.~Siligardi, J.~Panerati, M.~Kaufmann, M.~Minelli, C.~Ghedini, G.~Beltrame, and L.~Sabattini, ``Robust area coverage with connectivity maintenance,'' in \emph{2019 International Conference on Robotics and Automation (ICRA)}, 2019, pp. 2202--2208.

\bibitem{similar_problem}
S.~Kawajiri, K.~Hirashima, and M.~Shiraishi, ``Coverage control under connectivity constraints,'' in \emph{Proceedings of the 20th International Conference on Autonomous Agents and MultiAgent Systems}, ser. AAMAS '21.\hskip 1em plus 0.5em minus 0.4em\relax Richland, SC: International Foundation for Autonomous Agents and Multiagent Systems, 2021, p. 1554–1556.

\bibitem{control_barrier}
\BIBentryALTinterwordspacing
P.~Ong, B.~Capelli, L.~Sabattini, and J.~Cortés, ``Nonsmooth control barrier function design of continuous constraints for network connectivity maintenance,'' \emph{Automatica}, vol. 156, p. 111209, 2023. [Online]. Available: \url{https://www.sciencedirect.com/science/article/pii/S0005109823003709}
\BIBentrySTDinterwordspacing

\bibitem{cbf}
B.~Capelli and L.~Sabattini, ``Connectivity maintenance: Global and optimized approach through control barrier functions,'' in \emph{2020 IEEE International Conference on Robotics and Automation (ICRA)}, 2020, pp. 5590--5596.

\bibitem{voronoi_with_connectivity}
W.~Luo and K.~Sycara, ``Voronoi-based coverage control with connectivity maintenance for robotic sensor networks,'' in \emph{2019 International Symposium on Multi-Robot and Multi-Agent Systems (MRS)}, 2019, pp. 148--154.

\bibitem{laplacian}
N.~Biggs, \emph{Algebraic Graph Theory}, 2nd~ed., ser. Cambridge Mathematical Library.\hskip 1em plus 0.5em minus 0.4em\relax Cambridge University Press, 1974.

\bibitem{Bertsekas}
D.~P. Bertsekas, \emph{Nonlinear programming}.\hskip 1em plus 0.5em minus 0.4em\relax Athena Scientific, 1999.

\bibitem{MFCQ}
\BIBentryALTinterwordspacing
O.~Mangasarian and S.~Fromovitz, ``The fritz john necessary optimality conditions in the presence of equality and inequality constraints,'' \emph{Journal of Mathematical Analysis and Applications}, vol.~17, no.~1, pp. 37--47, 1967. [Online]. Available: \url{https://www.sciencedirect.com/science/article/pii/0022247X67901631}
\BIBentrySTDinterwordspacing

\bibitem{horn2012matrix}
R.~A. Horn and C.~R. Johnson, \emph{Matrix analysis}.\hskip 1em plus 0.5em minus 0.4em\relax Cambridge university press, 2012.

\end{thebibliography}

\end{document}